\documentclass[a4paper,reqno,11pt, oneside]{amsart}
\usepackage[pdftex]{graphicx,xcolor}
\usepackage{amssymb}
\usepackage{amstext}
\usepackage{amsmath}
\usepackage{amscd}
\usepackage{amsthm}
\usepackage{amsfonts}
\usepackage{enumerate}
\usepackage{caption}
\usepackage{ulem}
\usepackage{color}
\usepackage{here} 
\usepackage{bm} 
\usepackage{tikz}
\usepackage{calc}
\usetikzlibrary{decorations.markings}
\usetikzlibrary{positioning}
\tikzstyle{vertex}=[circle ,draw, inner sep=0pt, minimum size=6pt]

\makeatletter
  
  \@addtoreset{equation}{section}
\makeatother

\newcommand{\Ac}{\mathcal{A}}

\newcommand{\Lc}{\mathcal{L}}

\newcommand{\ZZ}{\mathbb{Z}}

\newcommand{\RR}{\mathbb{R}}

\newcommand{\kk}{\Bbbk}
\newcommand{\pf}{\mathfrak{p}}

\newcommand{\ov}{\overline}

\newcommand{\eb}{\mathbf{e}}

\newcommand{\xb}{\mathbf{x}}

\newcommand{\Ext}{\operatorname{Ext}}

\def\opn#1#2{\def#1{\operatorname{#2}}} 
\opn\conv{conv} \opn\dep{depth} \opn\Spec{Spec} \opn\cone{cone} \opn\ini{in} \opn\codeg{codeg} \opn\deg{deg}
\opn\Graph{Graph} \opn\sign{sign} \opn\Ehr{Ehr} \opn\rank{rank} \opn\type{type} \opn\reg{reg} \opn\core{core}
\opn\Hilb{Hilb} \opn\Indeg{Indeg} \opn\link{link} \opn\Tor{Tor} \opn\MNF{MNF} \opn\dep{depth} \opn\pdim{pdim}

\newtheorem{thm}{Theorem}[section]

\newtheorem{prop}[thm]{Proposition}

\theoremstyle{definition}

\newtheorem{ex}[thm]{Example}

\theoremstyle{remark}
\newtheorem{rem}[thm]{Remark}

\begin{document}

\title{Difference of Hilbert series of homogeneous monoid algebras and their normalizations}
\author{Akihiro Higashitani}

\address[A. Higashitani]{Department of Pure and Applied Mathematics, 
Graduate School of Information Science and Technology, 
Osaka University, Suita, Osaka 565-0871, Japan}
\email{higashitani@ist.osaka-u.ac.jp}


\maketitle

\begin{abstract} 
Let $Q$ be an affine monoid, $\kk[Q]$ the associated monoid $\kk$-algebra, and $\kk[\ov{Q}]$ its normalization, where we let $\kk$ be a field.  
In this paper, in the case where $\kk[Q]$ is homogeneous (i.e., standard graded), a difference of the Hilbert series of $\kk[Q]$ and $\kk[\ov{Q}]$ is discussed. 
More precisely, we prove that if $\kk[Q]$ satisfies Serre's condition $(S_2)$, then the degree of the $h$-polynomial of $\kk[Q]$ is 
always greater than or equal to that of $\kk[\ov{Q}]$. 
Moreover, we also show counterexamples of this statement if we drop the assumption $(S_2)$. 
\end{abstract}

\section{Introduction}

\subsection{Backgrounds}
Let $\kk$ be a field throughout this paper. 
For the fundamental materials on commutative algebra, see \cite{BH}. 

Let $R=\bigoplus_{i \in \ZZ}R_i$ be a graded $\kk$-algebra with $\dim_\kk R_i< \infty$ for each $i$. 
Hilbert series of $R$ is one of the most fundamental invariants in the theory of commutative algebra. 
For the introduction to the theory of Hilbert series, see, e.g., \cite[Section 4]{BH}. 
Although the Hilbert series just provides the numerical information of $R$ as a $\kk$-vector space, it reflects several commutative-algebraic properties. 
For example, a classical observation claims that if $R$ is Cohen--Macaulay, 
then the Laurent polynomial appearing in the numerator of the Hilbert series of $R$ has nonnegative coefficients (\cite[Corollary 4.1.10]{BH}). 
Moreover, the Gorensteinness of $R$ is completely characterized in terms of a kind of symmetry of the Hilbert series if $R$ is a domain (\cite[Corollary 4.4.6]{BH}). 
Furthermore, by several recent studies, some connections with certain generalized notions of Gorensteinness (e.g. almost Gorenstein, nearly Gorenstein, level) 
and the Hilbert series have been discovered. (See, e.g., \cite{H, HY, M1}.) 
Since the Hilbert series of $R$ can be computed from the graded minimal free resolution of $R$, 
it also captures other invariants which have been well-discussed in commutative algebra, such as  
Krull dimension (or codimension), Koszulness, Castelnuovo-Mumford regularity, multiplicity, $a$-invariant, and so on. 

On the other hand, affine monoids and their associated $\kk$-algebras have been well studied by many researchers in several contexts. 
For the introduction to the theory of affine monoids and affine monoid $\kk$-algebras, see, e.g., \cite{BG} and \cite[Section 6.1]{BH}. 
Since geometric information of affine monoids is applicable for the analysis of algebraic properties of the associated $\kk$-algebras, 	
affine monoid $\kk$-algebras have been regarded as useful objects in commutative algebra. 
For example, the celebrated theorem by Hochster claims that if an affine monoid $Q$ is normal, i.e., $\ov{Q}=Q$, 
then $\kk[Q]$ is Cohen--Macaulay (\cite[Theorem 6.3.5 (a)]{BH}). 
Moreover, Cohen--Macaulayness of the affine monoid $\kk$-algebra $\kk[Q]$ is characterized in terms of geometric information on $Q$ 
together with the information about reduced homology groups over $\kk$ of certain simplicial complexes (\cite{TH}). 
Furthermore, Katth\"an reveals a strong connection with the structure of holes of affine monoids $Q$ (i.e., $\ov{Q} \setminus Q$) 
and ring-theoretic properties (e.g., Serre's condition $(S_2)$, $(R_1)$ and depth). See \cite{K}. 

By taking these backgrounds into account, in this paper, we study the Hilbert series of $\kk[Q]$ and $\kk[\ov{Q}]$. 
In particular, we focus on the difference of them. 


\subsection{Main Results}

To explain the main results of this paper, we introduce the notations used throughout this paper. 
Given a graded $\kk$-algebra $R=\bigoplus_{i \in \ZZ}R_i$ with $\dim_\kk R_i<\infty$ for each $i$, 
let $\Hilb(R,t)$ denote the Hilbert series of $R$, i.e., $$\Hilb(R,t)=\sum_{i \in \ZZ}\dim_\kk R_i t^i.$$
We say that $R$ is \textit{homogeneous} (or \textit{standard graded}) if it is generated by degree $1$ elements. 
If $R$ is homogeneous and of dimension $d$, then we see that $\Hilb(R,t)$ is of the following form: $$\Hilb(R,t)=\frac{h_R(t)}{(1-t)^d},$$
where $h_R(t)$ is a polynomial in $t$ with integer coefficients. We call this polynomial $h_R(t)$ the \textit{$h$-polynomial} of $R$.

When $R=\kk[Q]$ for some affine monoid $Q$, we call $Q$ \textit{homogeneous} if $\kk[Q]$ is homogeneous. 
We use the notation $\Hilb(Q,t)$ and $h_Q(t)$ instead of $\Hilb(\kk[Q],t)$ and $h_{\kk[Q]}(t)$, respectively. 

\smallskip

The following is the first main theorem of this paper: 
\begin{thm}\label{main1}
Let $Q$ be a homogeneous affine monoid and assume that $\kk[Q]$ satisfies Serre's condition $(S_2)$. 
Then $\deg (h_Q(t)) \geq \deg (h_{\ov{Q}}(t))$. 
\end{thm}
Here, $\deg(f(t))$ denotes the degree of the polynomial $f(t)$. For the definition of Serre's condition, see Subsection~\ref{subsec:Serre}. 

The following second main theorem shows the existence of counterexamples of Theorem~\ref{main1} if we drop the assumption $(S_2)$. 
\begin{thm}\label{main2}
For any positive integer $m$, there exists a homogeneous affine monoid $Q$ such that $\deg(h_{\ov{Q}}(t)) - \deg(h_Q(t))=m$. 
\end{thm}

\subsection{Organization of this paper}

In Section~\ref{sec:pre}, we prepare the materials for the proofs of the main results. 
In Section~\ref{sec:main1}, we give a proof of Theorem~\ref{main1}. 
In Section~\ref{sec:main2}, we give a proof of Theorem~\ref{main2}. 

\subsection*{Acknowledgements}
This paper is partially supported by KAKENHI 20K03513 and 21KK0043. 

\medskip

\section{Preliminaries}\label{sec:pre}

In this section, we recall several materials used in this paper. 

\subsection{Affine monoids and hole modules}
An \textit{affine monoid} is a finitely generated submonoid of $\ZZ^d$ for some $d$. 
Given an affine monoid $Q \subset \ZZ_{\geq 0}^d$, we can associate the $\kk$-algebra $\kk[Q] \subset \kk[x_1,\ldots,x_d]$ defined by 
$$\kk[Q]=\kk[\xb^\alpha : \alpha \in Q],$$
where for $\alpha=(\alpha_1,\ldots,\alpha_d) \in \ZZ_{\geq 0}^d$, we let $\xb^\alpha=x_1^{\alpha_1}\cdots x_d^{\alpha_d}$. 
We call this $\kk$-algebra $\kk[Q]$ the \textit{monoid algebra} of $Q$. 
Affine monoids and monoid algebras have been called as affine semigroups and affine semigroup rings (e.g., in \cite{BH}), 
but it is becoming to be called them as affine monoids and monoid algebras, respectively. 
Those are the same notions, but we employ the terminology ``monoid'' in this paper. 

We recall some fundamental notions on affine monoids and their monoid $\kk$-algebras. 
Let $Q \subset \ZZ^d_{\geq 0}$ be an affine monoid. The minimal generating system of $Q$ is the minimal finite subset $\{\alpha_1,\ldots,\alpha_s\}$ of $Q$ 
such that $\displaystyle Q=\left\{\sum_{i=1}^s n_i\alpha_i : n_i \in \ZZ_{\geq 0}\right\}$. 
In this case, we use the notation $Q=\langle \alpha_1,\ldots,\alpha_s \rangle$. 

Let $Q=\langle \alpha_1,\ldots,\alpha_s \rangle$ for some $\alpha_i \in \ZZ_{\geq 0}^d$. 
\begin{itemize}
\item Let $\ZZ Q$ denote the free abelian group generated by $Q$, i.e., 
$\displaystyle \ZZ Q=\left\{\sum_{i=1}^s z_i \alpha_i : z_i \in \ZZ\right\}.$
\item Let $\RR_{\geq 0}Q$ denote the polyhedral cone generated by $Q$, i.e., 
$\displaystyle\RR_{\geq 0}Q = \left\{\sum_{i=1}^s r_i \alpha_i : r_i \in \RR_{\geq 0}\right\}.$
\item Let $\ov{Q}$ be the \textit{normalization} of $Q$, i.e., $$\ov{Q}=\ZZ Q \cap \RR_{\geq 0}Q.$$
Clearly, we have $Q \subset \ov{Q}$. We say that $Q$ is \textit{normal} if $Q=\ov{Q}$ holds. 
Note that $\kk[\ov{Q}]$ coincides with the normalization of $\kk[Q]$. 
It is known that $\kk[Q]$ is Cohen--Macaulay if $Q$ is normal. (See \cite[Theorem 6.3.5]{BH}.) 
\item A \textit{face} $F$ of $Q$ is a subset of $Q$ satisfying the following: $\alpha,\beta \in Q$, 
$\alpha + \beta \in F \Rightarrow \alpha \in F \text{ and }\beta \in F.$
The dimension of a face $F$ is defined to be the rank of the free abelian group $\ZZ F$. 
\item We say that $Q$ is \textit{positive} if the minimal face of $Q$ is $\{0\}$. 
\item Regarding $\kk[Q]$, we have $\dim \kk[Q]=\dim Q$. 
\item We say that an affine monoid $Q$ is \textit{homogeneous} if the minimal generating set of $Q$ lies on the same hyperplane not containing the origin. 
\end{itemize}
Throughout this paper, affine monoids are always assumed to be positive. 
We recall the following statement, which will play a crucial role in the proof of Theorem~\ref{main1}. 
\begin{thm}[{\cite[Theorem 3.1 and Proposition 5.5]{K}}]\label{thm:hole}
Let $Q$ be an affine monoid. Then there exists a (not-necessarily disjoint) decomposition: 
\begin{align}\label{eq:decomp}\ov{Q} \setminus Q=\bigcup_{i=1}^\ell (q_i + \ZZ F_i) \cap \RR_{\geq 0}Q\end{align}
with $q_i \in \ov{Q}$ and faces $F_i$ of $Q$. 

Moreover, in the case $\ell=1$, i.e., if $\ov{Q} \setminus Q = (q + \ZZ F) \cap \RR_{\geq 0}Q$ for some $q \in \ov{Q}$ and a face $F$ of $Q$, 
then $\dep \kk[Q] = \dim F +1$. 
\end{thm}
We call $q_i+\ZZ F_i$ appearing in \eqref{eq:decomp} a \textit{$j$-dimensional family of holes} of $Q$ if $\dim F_i=j$.

Note that Theorem 3.1 and Proposition 5.5 in \cite{K} claim the similar statement for non-necessarily positive affine monoids, 
but we convert the statement into the case of positive affine monoids. 

\subsection{Serre's condition}\label{subsec:Serre}
Let $R$ be a Noetherian ring and let $M$ be a finitely generated $R$-module. 
For a nonnegative integer $m$, we say that $M$ satisfies \textit{Serre's condition $(S_m)$} if 
$$\dep M_\pf \geq \min\{m,\dim M_\pf\} \text{ for all }\pf \in \Spec R,$$
where $\Spec R$ denotes the set of all prime ideals of $R$. 
Namely, $M$ satisfies $(S_m)$ if and only if $M_\pf$ is Cohen--Macaulay for any $\pf \in \Spec R$ with $\dep M_\pf <m$. 
\begin{prop}[{cf. \cite{Dao}}]\label{prop:dao}
Let $A=\kk[x_1,\ldots,x_n]$, let $I$ be an ideal of $A$ and let $R=A/I$. 
Then $R$ satisfies $(S_m)$ if and only if $\dim(\Ext^i_A(R,A)) \leq n-m-i$ holds for any $i > d-\dim R$. 
\end{prop}

Regarding $(S_2)$ for monoid algebras $\kk[Q]$, we konw the following: 
\begin{thm}[{\cite[Theorem 5.2]{K}}]\label{thm:S2}
Let $Q$ be a positive affine monoid of dimension $d$. 
Then $\kk[Q]$ satisfies $(S_2)$ if and only if every family of holes of $Q$ is of dimension $d-1$. 
\end{thm}

\subsection{Ehrhart rings of lattice polytopes} 
A convex polytope $P \subset \RR^N$ is called a \textit{lattice polytope} if all of its vertices belong to $\ZZ^N$. 
If $P \not\subset \RR_{\geq 0}^N$, then we may translate $P$ by a certain lattice point $\gamma \in \ZZ_{\geq 0}^N$ to make $P+\gamma \subset \RR_{\geq 0}^N$. 
Since most properties of lattice polytopes are preserved by the traslation by a lattice point, 
we may assume that $P \subset \RR_{\geq 0}^N$ without loss of generality. 

Given a lattice polytope $P \subset \RR_{\geq 0}^N$, we can associate a homogeneous affine monoid $Q_P$ as follows: 
$$Q_P =\langle (\alpha,1) : \alpha \in P \cap \ZZ^N \rangle \subset \RR^{N+1}.$$
Hence, we can associate the $\kk$-algebra $\kk[Q_P]$, known as the \textit{toric ring} of $P$. 

On the other hand, we can also define the $\kk$-algebra associated to $P$ as follows: 
$$\Ehr_\kk(P):=\kk[{\bf x}^\alpha x_{N+1}^n : \alpha \in nP \cap \ZZ^N] \subset \kk[x_1,\ldots,x_{N+1}],$$
where $nP=\{n \alpha : \alpha \in P\}$. 
This $\kk$-algebra $\Ehr_\kk(P)$ is called the \textit{Ehrhart ring} of $P$. 
We see by definition that $\Hilb(\Ehr_\kk(P),t)$ coincides with the \textit{Ehrhart series} of $P$, 
which is the generating function $\sum_{n \geq 0}|nP \cap \ZZ^N|t^n$. 
For the introduction to Ehrhart theory, see \cite{BR}. 

Given a homogeneous affine monoid $Q \subset \ZZ_{\geq 0}^d$, a \textit{cross section polytope} of $Q$, denoted by $P_Q \subset \RR_{\geq 0}^d$, 
is the lattice polytope obtained by the intersection of $Q$ and a hyperplane including all the generators of $Q$.

\begin{ex}
The Ehrhart ring of $P$ does not necessarily coincide with the normalization of $\kk[Q_P]$. 
In fact, let $P$ be the convex hull of $\{(0,0,0),(1,1,0),(0,1,1),(1,0,1)\}$. Then we see that 
$$\kk[Q_P] = \kk[x_4,x_1x_2x_4,x_2x_3x_4,x_1x_3x_4] \cong \kk[X_1,\ldots,X_4].$$ 
In particular, the normalization $\kk[\ov{Q_P}]$ is its own. 
On the other hand, the following holds: 
\begin{align*}
\Ehr_\kk(P)&=\kk[x_4,x_1x_2x_4,x_2x_3x_4,x_1x_3x_4, \uwave{x_1x_2x_3x_4^2}] \\
&\cong \kk[X_1,\ldots,X_4,Y]/(X_1X_2X_3X_4-Y^2),
\end{align*}
where we regard $\deg X_i=1$ and $\deg Y=2$. 
\end{ex}
We say that a lattice polytope $P \subset \RR^d_{\geq 0}$ is \textit{spanning} if $\ZZ Q_P =\ZZ^{d+1}$. 
If $P$ is spanning, then we see that $\kk[\ov{Q_P}]=\Ehr_\kk(P)$. 
The notion of spanning polytopes was introduced in \cite{HKN} and the Ehrhart theory for spanning polytopes was developed there.

\smallskip

We recall a well-known notion for lattice polytopes $P \subset \RR^d$. 
The \textit{codegree of $P$}, denoted by $\codeg(P)$, is defined as follows: 
\begin{align}\label{eq:codeg}
\codeg(P):=\min\{\ell \in \ZZ_{>0} : \ell P^\circ \cap \ZZ^d \neq \emptyset\},
\end{align}
where $P^\circ$ denotes the relative interior of $P$. 

Assume that $P \subset \RR_{\geq 0}^d$ and $P$ is spanning. 
Then it is known that $$\deg(h_{\ov{Q_P}}(t))=\dim \kk[\ov{Q_P}] - \codeg(P)$$
and $$h_s=|\ell P^\circ \cap \ZZ^d|,$$
where $s=\deg (h_{\ov{Q_P}}(t))$, $h_s$ is the leading coefficient of $h_{\ov{Q_P}}(t)$ and $\ell=\codeg(P)$. 
Those are implicitly explained in \cite[Section 4]{BR} in terms of lattice polytopes or their Ehrhart series. 

\subsection{Edge rings}

Throughout this paper, all graphs are finite and simple. 
Let $G$ be a connected graph on the vertex set $V(G)$ with the edge set $E(G)$. 
Then we can associate a homogeneous affine monoid $Q_G$ by setting 
$$Q_G=\langle \eb_{u,v} : \{u,v\} \in E(G) \rangle \subset \RR^{V(G)},$$
where $\eb_u$ denotes the unit vector of $\RR^{V(G)}$ and $\eb_{u,v}=\eb_u+\eb_v$. 
We call the associated $\kk$-algebra $\kk[Q_G]$ the \textit{edge ring} of $G$. 

Let $P_G$ be the convex hull of $\{\eb_{u,v} : \{u,v\} \in E(G)\}$. 
This polytope is known as the \textit{edge polytope} of $G$. 

We know by \cite[Proposition 1.3]{OH} that 
$$\dim \kk[G] = \dim P_G+1 = \begin{cases} |V(G)|-1, &\text{if $G$ is bipartite}, \\ |V(G)|, &\text{otherwise}. \end{cases}$$

\begin{rem}\label{rem}
For a non-bipartite connected graph $G$ with $d$ vertices, the edge polytope $P_G$ is always assumed to be spanning in the following sense. 
Fix a vertex $v$ of $G$. Let $P'$ be the image of $P_G$ by the projection $\pi:\RR^{V(G)} \rightarrow \RR^{V(G) \setminus \{v\}}$ 
ignoring the entry corresponding to $v$. Then $Q_G$ is isomorphic to $Q_{P'}$ as affine monoids. 
Moreover, let $S$ be a $(d-1)$-simplex in $P_G$ described in \cite[Lemma 1.4]{OH}. 
Then we can check that $\{(\alpha,1) : \alpha \in \pi(S) \cap \ZZ^{V(G) \setminus \{v\}}\}$ forms a $\ZZ$-basis for $\ZZ^{V(G)}$ 
if we choose a vertex $v$ properly. Therefore, we conclude that $P'$ is spanning. 
Since $Q_G$ is isomorphic to $Q_{P'}$, we can claim like ``$P_G$ is spanning''. In particular, $\kk[\ov{Q_G}]=\Ehr_\kk(P_G)$ holds. 
Even if $G$ is not connected, we may apply the same procedure for each connected component $G_i$ and obtain a $\ZZ$-basis for $\ZZ^{V(G_i)}$. 
By combining these $\ZZ$-bases for $\ZZ^{V(G_i)}$ for each $i$, we obtain a $\ZZ$-basis for $\ZZ^{V(G)}$. 
Namely, $P_G$ is spanning for any non-bipartite (non-necessarily connected) graph $G$. 

A similar discussion can be applied for bipartite graphs and we can also claim the spanning property for edge polytopes of bipartite graphs, 
but we omit the detail since we do not use it in this paper.
\end{rem}

An \textit{exceptional pair} in $G$ is a pair $(C,C')$ of two odd cycles $C$ and $C'$ in $G$ such that 
$C$ and $C'$ have no common vertex and there is no bridge between $C$ and $C'$ (i.e., no edge $\{v,v'\}$ in $G$ with $v \in V(C)$ and $v' \in V(C')$). 

Edge rings have been intensively studied by several people since the following theorem was established: 
\begin{thm}[{\cite{OH, SVV}}]
Let $G$ be a connected graph. Then $\ov{Q_G}$ is described as follows: 
\begin{align}\label{eq:exc}\ov{Q_G}=Q_G+\ZZ_{\geq 0}\{\eb_C+\eb_{C'} : (C,C') \text{ is an exceptional pair in }G\},\end{align}
where $\eb_C=\sum_{v \in V(C)}\eb_v$. In particlar, $Q_G$ is normal if and only if there is no exceptional pair in $G$. 
\end{thm}

We prepare the following proposition for the proof later. 
\begin{prop}\label{prop:deg_edge}
For a connected graph with $d$ vertices, we have $\deg (h_{\ov{Q_G}}(t)) \leq d/2$. 
Moreover, if the equality holds, then $h_{d/2}=1$, where $h_{d/2}$ is the leading coefficient of $h_{\ov{Q_G}}(t)$.  
\end{prop}
\begin{proof}
For the first statement, it is enough to show that $\codeg(P_G) \geq d/2$ by \eqref{eq:codeg} and Remark~\ref{rem}. 

Let $\ell$ be a positive integer and assume that $\ell P_G^\circ \cap \ZZ^d \neq \emptyset$. 
Let $\alpha=(\alpha_v)_{v \in V(G)} \in \ell P_G^\circ \cap \ZZ^d$. 
Then we can write $\alpha$ like $\alpha=\sum_{\{u,v\} \in E(G)}a_{u,v}\eb_{u,v}$, 
where $a_{u,v}>0$ for each $\{u,v\} \in E(G)$ and $\sum a_{u,v}=\ell$. 
Thus, in particular, $\alpha_u \geq 1$ holds for each $u \in V(G)$ since $\alpha \in \ZZ^d$. 
Hence, $$2\ell =2\sum a_{u,v}=\sum_{v \in V(G)}\alpha_v \geq |V(G)| = d.$$ This means that $\codeg(P_G) \geq d/2$. 

Moreover, by this discussion, if $\codeg(P_G)=d/2$, then $\alpha$ must be $\sum_{v \in V(G)}\eb_u$. This implies that $h_{d/2}=1$. 
\end{proof}

\medskip

\section{Proof of Theorem~\ref{main1}}\label{sec:main1}

This section is devoted to giving a proof of Theorem~\ref{main1}. 

\begin{proof}[Proof of Theorem~\ref{main1}]
Let $Q \subset \ZZ_{\geq 0}^d$ be a homogeneous affine monoid of dimension $d$. 
Then we know the following short exact sequence of graded $\kk[Q]$-modules: 
\begin{align*}
0 \longrightarrow \kk[Q] \longrightarrow \kk[\ov{Q}] \longrightarrow \kk[\ov{Q}]/\kk[Q] \longrightarrow 0
\end{align*}
Thus, we have \begin{align}\label{eq:sum}\Hilb(Q,t)=\Hilb(\ov{Q},t) - \Hilb(\kk[\ov{Q}]/\kk[Q],t).\end{align} 

By the way, Theorem~\ref{thm:hole} describes the structure of $\kk[\ov{Q}]/\kk[Q]$. 
Since we assume $(S_2)$ for $\kk[Q]$, we know that each family of holes is of dimension $d-1$ (see Theorem~\ref{thm:S2}). 
Let us consider the decomposition of $\ov{Q} \setminus Q$ in \eqref{eq:decomp}, where $q_i \in \ov{Q}$ and $F_i$ is a face of $Q$ of dimension $d-1$ for each $i$. 
Note that this decomposition is not necessarily disjoint. Hence, we have to apply an ``inclusion-exclusion type'' formula to get the Hilbert series of $\kk[\ov{Q}]/\kk[Q]$. 

Here, we recall some materials on hyperplanes arrangements. (See, e.g., \cite{S} for the introduction to the theory of hyperplane arrangements.) 
Let $H_i=q_i+\RR F_i$ be an affine hyperplane, and let $\Ac=\{H_1,\ldots,H_\ell\}$ be a hyperplane arrangement. 
Let $\Lc(\Ac)$ denote the intersection lattice of $\Ac$, i.e., $$\Lc(\Ac)=\left\{ \bigcap_{i \in I} H_i \neq \emptyset : I \subset \{1,\ldots,\ell\} \right\}$$ 
equipped with a partial order defined by reverse inclusion. 
We regard $\bigcap_{i \in \emptyset}H_i$ as $\RR^d$, i.e., $\RR^d \in \Lc(\Ac)$ is the minimal element. 
We define a map $\mu : \Lc(\Ac) \rightarrow \ZZ$ (known as the \textit{M\"obius function}) as follows: 
$$\mu(X)=\begin{cases}
1, \; &\text{ if }X=\RR^d, \\
-\sum_{X \subsetneq Y}\mu(Y), &\text{ otherwise}. 
\end{cases}$$

Then we see from \eqref{eq:decomp} that 
\begin{align*}
\Hilb(\kk[\ov{Q}]/\kk[Q],t)&=-\sum_{X \in \Lc(\Ac) \setminus \RR^d} \mu(X)\frac{t^{a_i}h_{\ov{X}}(t)}{(1-t)^{\dim X}} \\
&=-\sum_{X \in \Lc(\Ac) \setminus \RR^d} \mu(X)\frac{t^{a_i}h_{\ov{X}}(t)(1-t)^{d-\dim X}}{(1-t)^d},
\end{align*}
where $a_i=\deg \xb^{q_i}$ and we let $\ov{X}=X \cap \ov{Q}$. Note that $\ov{X}$ is a normal homogeneous submonoid of $\ov{Q}$. 
Hence, it follows from \eqref{eq:sum} that 
$$h_Q(t)=h_{\ov{Q}}(t)+\sum_{X \in \Lc(\Ac)}(-1)^{d - \dim X}\mu(X)t^{a_i}(t-1)^{d - \dim X}h_{\ov{X}}(t).$$
Since $\ov{X}$ is normal, we know that $\kk[\ov{X}]$ is Cohen--Macaulay, so $h_{\ov{X}}(t)$ has positive coefficients. 
Thus, we can conclude the desired conclusion $\deg (h_Q(t)) \geq \deg (h_{\ov{Q}}(t))$ once we check that $(-1)^{d- \dim X}\mu(X) \geq 0$. 
Actually, this is known to be true. See, e.g., \cite[Corollary 3.4]{S}. 
\end{proof}

Regarding another relationship between $h_Q(t)$ and $h_{\ov{Q}}(t)$, we know the following: 
\begin{prop}\label{prop:mul}
For any homogeneous affine monoid $Q$ of dimension $d$, $h_Q(1)=h_{\ov{Q}}(1)$ holds. 
\end{prop}
\begin{proof}
Since the structure of $\kk[\ov{Q}]/\kk[Q]$ can be captured by the decomposition \eqref{eq:decomp} and each face $F_i$ is of dimension at most $d-1$, 
we see that $$\Hilb(\kk[\ov{Q}]/\kk[Q],t)=\frac{g(t)}{(1-t)^{d-1}},$$
where $g(t)$ is some polynomial in $t$. Then it follows from \eqref{eq:sum} that 
\begin{align*}
\Hilb(\kk[\ov{Q}],t)-\Hilb(\kk[Q],t)=\frac{h_{\ov{Q}}(t)-h_Q(t)}{(1-t)^d}=\frac{(1-t)g(t)}{(1-t)^d}. 
\end{align*}
In particular, $h_{\ov{Q}}(t)-h_Q(t)$ is divisible by $(1-t)$. This means that $h_{\ov{Q}}(1)=h_Q(1)$. 
\end{proof}

\medskip

\section{Proof of Theorem~\ref{main2}}\label{sec:main2}

In this section, we give a proof of Theorem~\ref{main2}. 
For the proof of Theorem~\ref{main2}, we use the following example. 
\begin{ex}\label{ex:Q}
Let $G$ be the graph on $\{1,\ldots,10\}$ with the edge set 
$$E(G)=\{12,13,23,34,35,36,37,45,67,58,78,89, 8 \, 10,9 \, 10\}. $$
See Figure~\ref{graph1}. 
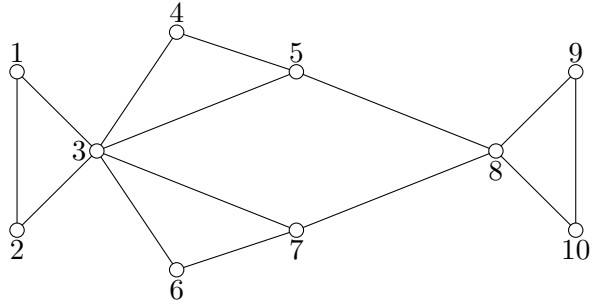
\begin{figure}[h]
\centering
\begin{tikzpicture}[scale=1.05]
\coordinate(3)at(0,0);
\coordinate(1)at(-1,1);
\coordinate(2)at(-1,-1);
\coordinate(4)at(1,1.5);
\coordinate(6)at(1,-1.5);
\coordinate(5)at(2.5,1);
\coordinate(7)at(2.5,-1);
\coordinate(8)at(5,0);
\coordinate(9)at(6,1);
\coordinate(10)at(6,-1);
\draw(7)--(8)--(10)--(9)--(8)--(5)--(3)--cycle;
\draw(5)--(4)--(3)--(1)--(2)--(3)--(6)--(7);
\foreach \i in {1,2,...,10}{
\filldraw[fill=white](\i)circle(0.09);
}
\foreach \i in {2,6,7,8,10}{
\draw(\i)node[below]{\i};
}
\foreach \i in {1,4,5,9}{
\draw(\i)node[above]{\i};
}
\draw(3)node[left]{3};
\end{tikzpicture}
\caption{A graph $G$ with $\deg (h_{Q_G}(t)) < \deg (h_{\ov{Q_G}}(t))$}\label{graph1}
\end{figure}

Let $Q=Q_G$. By using {\tt Macaulay2} (for $\kk[Q]$) together with {\tt Normaliz} (for $\kk[\ov{Q}]$), we see the following: 
\begin{align*}
\Hilb(Q,t)=\frac{1+4t+9t^2+12t^3+8t^4}{(1-t)^{10}} \;\;\text{and}\;\; \Hilb(\ov{Q},t)=\frac{1+4t+9t^2+13t^3+6t^4+t^5}{(1-t)^{10}}. 
\end{align*}
On the other hand, we see that $$\dim \Ext^5_R(\kk[Q],A)=8>14-2-5, \;\; 5>14-\dim \kk[Q],$$
where we let $A=\kk[x_1,\ldots,x_{14}]$ and regard $\kk[Q]$ as $A/I$ by taking the defining ideal $I \subset A$ of $\kk[Q]$ (i.e., the toric ideal of $Q$). 
This implies that $\kk[Q]$ does not satisfy $(S_2)$ (see Proposition~\ref{prop:dao}). 
Note that we can also check non-$(S_2)$-ness by using the structure of the hole module $\kk[\ov{Q}]/\kk[Q]$. See Proposition~\ref{prop:Q_k}. 

This example shows that Theorem~\ref{main1} does not hold if we drop the assumption $(S_2)$. 
\end{ex}
Actually, we can generalize this example as follows: 
\begin{prop}\label{prop:Q_k}
Given a positive integer $k$, let $G_k$ be the graph as depicted in Figure~\ref{graph2}: \vspace{-0.5cm}
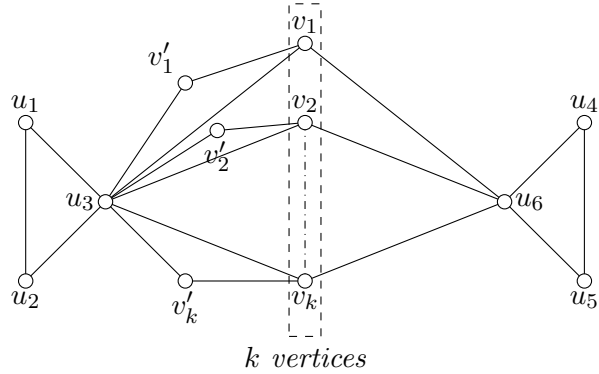
\begin{figure}[h]
\centering
\begin{tikzpicture}[scale=1.05]
\coordinate(3)at(0,0);
\coordinate(1)at(-1,1);
\coordinate(2)at(-1,-1);
\coordinate(4)at(1,1.5);
\coordinate(6)at(1,-1);
\coordinate(5)at(2.5,1);
\coordinate(12)at(2.5,2);
\coordinate(7)at(2.5,-1);
\coordinate(8)at(5,0);
\coordinate(9)at(6,1);
\coordinate(10)at(6,-1);
\coordinate(11)at(1.4,0.9);
\draw(7)--(8)--(10)--(9)--(8)--(5)--(3)--cycle;
\draw(4)--(3)--(1)--(2)--(3)--(6)--(7);
\draw(3)--(12)--(8);
\draw(3)--(11)--(5);
\draw(4)--(12);
\draw[dash dot](5)--(7);
\foreach \i in {1,2,...,12}{
\filldraw[fill=white](\i)circle(0.09);
}
\draw(1)node[above]{$u_{1}$};
\draw(2)node[below]{$u_{2}$};
\draw(3)node[left]{$u_{3}$};
\draw(4)node[above left]{$v_{1}'$};
\draw(5)node[above]{$v_{2}$};
\draw(6)node[below]{$v_{k}'$};
\draw(7)node[below]{$v_{k}$};
\draw(8)node[right]{$u_{6}$};
\draw(9)node[above]{$u_{4}$};
\draw(10)node[below]{$u_{5}$};
\draw(11)node[below]{$v_{2}'$};
\draw(12)node[above]{$v_{1}$};
\draw(2.5,-1.7)node[below]{$k$ vertices};
\draw[dashed](2.3,-1.7)--(2.7,-1.7)--(2.7,2.5)--(2.3,2.5)--cycle;
\end{tikzpicture}
\caption{Graph $G_k$}\label{graph2}
\end{figure}

\vspace{-0.4cm}

\noindent
Let $Q =Q_{G_k}$. Then 
\begin{align}\label{eq:con}\deg (h_Q(t)) - \deg (h_{\ov{Q}}(t))=
\begin{cases} -1 \;\;&\text{if $k$ is even}, \\
0 &\text{if $k$ is odd}. 
\end{cases}
\end{align}
We also have $\dep \kk[Q]=k+7$, while $\dim \kk[Q]=2k+6$. 
In particular, $\kk[Q]$ is never Cohen--Macaulay if $k \geq 2$. 
\end{prop}
\begin{proof}
Note that $\dim \kk[Q] = \dim \kk[\ov{Q}] = |V(G_k)|=2k+6$. 
Let $P=P_{G_k}$ be the edge polytope of $G_k$.

\noindent
\underline{The first step}: We compute $\deg(h_{\ov{Q}}(t))$ by computing the codegree of $P$. 
We see that 
\begin{align*}
&\frac{3}{4}\eb_{u_1,u_2}+\frac{1}{4}(\eb_{u_1,u_3}+\eb_{u_2,u_3})+\frac{7}{12}\eb_{u_4,u_5}+\frac{5}{12}(\eb_{u_4,u_6}+\eb_{u_5,u_6}) \\
&+\sum_{i=1}^k\frac{1}{6k}\left(\eb_{v_i,u_3}+\eb_{v_i,u_6}+2\eb_{v_i',u_3}+(6k-2)\eb_{v_i,v_i'}\right) =\sum_{v \in V(G_k)}\eb_v \in \RR^{V(G_k)}, 
\end{align*}
where $\eb_{u,v}=\eb_u+\eb_v$. 
In the left-hand side of this equality, all edges of $G_k$ appear and each of the coefficients is positive. 
Moreover the sum of the coefficients is equal to $k+3$. 
This implies that $(k+3)P^\circ \cap \ZZ^d \neq \emptyset$, i.e., $\codeg P \leq k+3$, 
which implies that $\deg (h_Q(t)) =\dim \kk[\ov{Q}] - \codeg P \geq k+3$. 
On the other hand, Proposition~\ref{prop:deg_edge} implies that $\deg (h_{\ov{Q}}(t)) \leq (2k+6)/2=k+3$. Therefore, $\deg (h_{\ov{Q}}(t))=k+3$. 
At the same time, we can also see that $h_{k+3}=1$, where $h_{k+3}$ denotes the leading coefficient of $h_{\ov{Q}}(t)$. 

\smallskip

\noindent
\underline{The second step}: 
Next, we compute the family of holes of $Q$. 
Let $G'$ be the (non-connected) graph obtained from $G_k$ by removing the vertices $v_1,\ldots,v_k$ together with the incident $3k$ edges. 
Let $q=\sum_{i=1}^6 \eb_{u_i}$ and let $Q'=Q_{G'}$. We claim that 
\begin{align}\label{eq:Q_k}\ov{Q} \setminus Q = q+Q'. \end{align}
Note that the pair of $3$-cycles $(u_1,u_2,u_3)$ and $(u_4,u_5,u_6)$ is a unique exceptional pair in $G_k$.

\noindent
``$(\subset)$'' By \eqref{eq:exc}, we have $\ov{Q} \setminus Q \subset Q + \ZZ_{\geq 0}q$, 
but we know that $$2q = \eb_{u_1,u_2}+\eb_{u_1,u_3}+\eb_{u_2,u_3}+\eb_{u_4,u_5}+\eb_{u_4,u_6}+\eb_{u_5,u_6} \in Q.$$
Hence, $\ov{Q} \setminus Q \subset q+Q$ holds. 
Moreover, for each $i=1,\ldots,k$, we see the following: 
\begin{align*}
q+\eb_{v_i,u_3}&=\eb_{u_1,u_3}+\eb_{u_2,u_3}+\eb_{v_i,u_6}+\eb_{u_4,u_5}, \\
q+\eb_{v_i,u_6}&=\eb_{u_1,u_2}+\eb_{v_i,u_3}+\eb_{u_4,u_6}+\eb_{u_5,u_6}, \text{ and}\\
q+\eb_{v_i,v_i'}&=\eb_{u_1,u_2}+\eb_{v_i',u_3}+\eb_{v_i4,u_6}+\eb_{u_4,u_5}. 
\end{align*}
This concludes that $\ov{Q} \setminus Q \subset q+Q'$. 

\noindent
``$(\supset)$'' Since we see from \eqref{eq:exc} that $q+Q' \subset \ov{Q}$ holds, 
it is enough to check $q+\alpha' \not\in Q$ for any $\alpha' \in Q'$. 
Let $(\alpha''_v)_{v \in V(G_k)}=q+\alpha'$ and look at the entries $\alpha_{u_4}'',\alpha_{u_5}''$ and $\alpha_{u_6}''$. 
Then $\alpha_{u_4}''+\alpha_{u_5}''+\alpha_{u_6}''$ is always odd (with at least $3$). 
On the other hand, $\alpha_{v_i}''=0$ for each $i$ by definition of $Q'$. 
This implies that $q+\alpha'$ cannot be decomposed into $\eb_{u,v}$'s for some edges $\{u,v\}$ in $G_k$, i.e., $q+\alpha' \not\in Q$.

Moreover, since $\{\eb_{u,v} : \{u,v\} \in G_k'\}$ is linearly independent and consists of $k+6$ vectors, we see that $\kk[Q'] \cong \kk[x_1,\ldots,x_{k+6}]$. 
Hence, $$\Hilb(Q',t)=\frac{1}{(1-t)^{k+6}}.$$

\smallskip

\noindent
\underline{The third step}: 
By the second step together with \eqref{eq:sum}, we see the following: 
\begin{align*}
\Hilb(Q,t)=\frac{h_{\ov{Q}}(t)}{(1-t)^{2k+6}}-t^3\frac{1}{(1-t)^{k+6}}=\frac{h_{\ov{Q}}(t)-t^3(1-t)^k}{(1-t)^{2k+6}}. 
\end{align*}
(Note that $q$ corresponds to a monomial of degree $3$ in $\kk[\ov{Q}]$.) 
Since $\deg (h_{\ov{Q}}(t))=k+3$ and $h_{k+3}=1$ by the first step, we obtain that 
$$\deg (h_{Q}(t))=\begin{cases}
k+2 \;\;&\text{if $k$ is even}, \\
k+3 &\text{if $k$ is odd}. 
\end{cases}$$
This implies the desired conclusion \eqref{eq:con}. 

\smallskip

\noindent
\underline{The fourth step}: Lastly, we discuss the depth of $\kk[Q]$. 
One can verify that $Q'$ is a face of $Q$ of codimension $k$. 
(In fact, $Q'=Q \cap \bigcap_{i=1}^k \{x_{v_i}=0\}$, where each $x_{v_i}=0$ becomes a supporting hyperplane of $Q$.) 
This implies that \eqref{eq:Q_k} gives the decomposition \eqref{eq:decomp}. 
Hence, $\kk[Q]$ does not satisfy $(S_2)$ by Theorem~\ref{thm:S2}. 
Moreover, we can apply the latter statement of Theorem~\ref{thm:hole} and obtain that the depth of $\kk[Q]$ is equal to $1+\dim Q' = k+7$. 
\end{proof}

For the proof of Theorem~\ref{main2}, we recall the notion of join for lattice polytopes and apply the same idea to homogeneous affine monoids. 
Given two lattice polytopes $P \subset \RR^d$ and $P \subset \RR^{d'}$, let $P \star P'$ 
be te convex hull of $$\{(\alpha,{\bf 0}_{d'},0) : \alpha \in P\} \cup \{({\bf 0}_d, \alpha,1) : \alpha' \in P'\}) \subset \RR^{d+d'+1}.$$
We call $P \star P'$ the \textit{join} of $P$ and $P'$. 

Similarly, given two homogeneous affine monoids $Q \subset \ZZ_{\geq 0}^d$ and $Q' \subset \ZZ_{\geq 0}^{d'}$, 
let $\alpha_1,\ldots,\alpha_s$ (resp. $\alpha_1',\ldots,\alpha_{s'}'$) be the minimal generating set of $Q$ (resp. $Q'$). 
We define $$Q \star Q'=\langle (\alpha_1,{\bf 0}_{d'},0),\ldots, (\alpha_s,{\bf 0}_{d'},0),({\bf 0}_d,\alpha_1',1),\ldots,({\bf 0}_d,\alpha_{s'}',1)\rangle \subset \ZZ_{\geq 0}^{d+d'+1}.$$
Then it is straightforward to see that $Q \star Q'$ is also a homogeneous affine monoid. 
Let us call $Q \star Q'$ the \textit{join} of homogeneous affine monoids $Q$ and $Q'$. 
\begin{prop}[{cf. \cite[Lemma 1.3]{HT}}]\label{prop:join}
Given two homogeneous affine monoids $Q \subset \ZZ_{\geq 0}^d$ and $Q' \subset \ZZ_{\geq 0}^{d'}$, 
we have $$\Hilb(Q \star Q',t) =  \Hilb(Q,t) \cdot \Hilb(Q',t)$$ and 
$$\Hilb(\ov{Q \star Q'},t) =  \Hilb(\ov{Q},t) \cdot \Hilb(\ov{Q'},t).$$ 
\end{prop}
\begin{proof}
Let $P$ (resp. $P'$) be the cross section polytope of $Q$ (resp. $Q'$). 
Then that of $Q \star Q'$ is nothing but $P \star P'$ since the hyperplnae containing the generators of $Q \star Q'$ is defined by 
$b'\sum_{i=1}^da_ix_i+b\sum_{i=1}^{d'}a_i'x_{d+i}=bb'$, where $\sum_{i=1}^da_ix_i=b$ (resp. $\sum_{i=1}^{d'}a_i'x_i=b'$) 
is the hyperplane containing the generators of $Q$ (resp. $Q'$). 
Hence, the latter equality of the statement is directly obtained from \cite[Lemma 1.3]{HT}. 

For the former, similarly to the proof of \cite[Lemma 1.3]{HT}, it suffices to show that 
$$\dim_\kk (\kk[Q \star Q']_k)=\sum_{i+j=k}\dim_\kk (\kk[Q])_i\dim_\kk(\kk[Q'])_j,$$
but this directly follows from the description of $Q \star Q'$ as follows: for each $k \in \ZZ_{>0}$, 
\begin{align*}
(Q \star Q') &\cap \left\{b'\sum_{i=1}^da_ix_i+b\sum_{i=1}^{d'}a_i'x_{d+i}=bb'k \right\} \\
&= \left\{m (\alpha,{\bf 0}_{d'},0) + (k-m) ({\bf 0}_d,\alpha',1) : \alpha \in Q, \alpha' \in Q', 0 \leq m \leq k, m \in \ZZ\right\}.
\end{align*}
\end{proof}

Now, we are in the position to give a proof of Theorem~\ref{main2}. 
\begin{proof}[Proof of Theorem~\ref{main2}]
Given a positive integer $m$, let $$Q^{(m)}=\underbrace{Q \star \cdots \star Q}_{m \text{ times}},$$
where $Q$ is the same homogeneous affine monoid as in Example~\ref{ex:Q}. 
It then follows from Proposition~\ref{prop:join} that $\Hilb(Q^{(m)},t)=\Hilb (Q,t)^m$ and $\Hilb(\ov{Q^{(m)}},t)=\Hilb(\ov{Q},t)^m$. 
Hence, $\deg(h_{\ov{Q^{(m)}}}(t))-\deg(h_{Q^{(m)}}(t))=m\left(\deg(h_{\ov{Q}}(t))-\deg(h_Q(t))\right)=m$, as required. 
\end{proof}

\begin{rem}
The difference of the degrees of $h_Q(t)$ and $h_{\ov{Q}}(t)$ can be arbitrarily large. 
For the case $\deg (h_Q(t)) < \deg(h_{\ov{Q}}(t))$, this is just what Theorem~\ref{main2} claims. 
For the case $\deg (h_Q(t)) \geq \deg(h_{\ov{Q}}(t))$, we can find such examples, e.g., in \cite[Theorem A]{M2}. 
In fact, for any nonnegative integer $m$, let $$R_m:=\langle (0,2m+3),(m+1,m+2),(m+2,m+1),(2m+3,0) \rangle \subset \ZZ_{\geq 0}^2.$$
Then $\kk[R_m]$ is nearly Gorenstein with Cohen--Macaulay type $2$ (see \cite[Lemma A ($\alpha$)]{M2}). 
Moreover, we also see that the projective dimension of $\kk[R_m]$ is $2$. 
Thus, by \cite[Theorem 4.10]{M1}, we know that $\displaystyle \Hilb(R_m,t)=\frac{1+2\sum_{i=1}^st^i}{(1-t)^2}$ for some $s$. 
In particular, $h_{R_m}(t)=1+2\sum_{i=1}^s t^i$. 
On the other hand, since $\ov{R_m}=\ZZ_{\geq 0}^2 \cap \{(x,y) : x+y \in (2m+3)\ZZ\}$, 
we also see that $\kk[\ov{R_m}] \cong \kk[X,Y]^{(2m+3)}$, where $\kk[X,Y]^{(n)}$ stands for the $n$-th Veronese subring of $\kk[X,Y]$. 
In particular, $\displaystyle \Hilb(\ov{R_m},t)=\frac{1+(2m+2)t}{(1-t)^2}$. 
Since $h_{R_m}(1)=h_{\ov{R_m}}(1)=2m+3$ (see Proposition~\ref{prop:mul}), we obtain that $s=m+1$. 
Therefore, $\deg(h_{R_m}(t))-\deg(h_{\ov{R_m}}(t))=m$ holds, as desired. 
\end{rem}


\medskip


\begin{thebibliography}{99}
\bibitem{BR} M. Beck and S. Robins, ``Computing the Continuous Discretely'',  Undergraduate Texts in Mathematics, Springer, 2007. 
\bibitem{BG} W. Bruns and J. Gubeladze, ``Polytopes, rings, and K-theory'', Springer Monographs in Mathematics. New York, NY, (2009). 
\bibitem{BH} W. Bruns and J. Herzog, ``Cohen-Macaulay rings, revised edition'', Cambridge University Press, 1998. 
\bibitem{Dao} H. Dao, Determining if a ring satisfies Serre's condition $S_{n}$, \textit{MathOverflow}, URL: {\tt https://mathoverflow.net/q/36957}. 
\bibitem{HT} M. Henk and M. Tagami, Lower bounds on the coefficients of Ehrhart polynomials, \textit{European J. Combin.} {\bf 30}, (2009), no.1, 70–-83. 
\bibitem{H} A. Higashitani, Almost Gorenstein homogeneous rings and their $h$-vectors, \textit{J. Algebra} \textbf{456} (2016), 190--206. 
\bibitem{HY} A. Higashitani and K. Yanagawa, Non-level semi-standard graded Cohen-Macaulay domain with $h$-vector $(h_0,h_1,h_2)$, \textit{J. Pure Appl. Algebra} \textbf{222} (2018), no.1, 191--201. 
\bibitem{HKN} J. Hofscheier, L. Katth\"an and B. Nill, Ehrhart theory of spanning lattice polytopes, \textit{Int. Math. Res. Not. IMRN} (2018), no.19, 5947–-5973.
\bibitem{K} L. Katth\"an, Non-normal affine monoid algebras, \textit{Manuscr. Math.} \textbf{146}, (2015), 223--233. 
\bibitem{M1} S. Miyashita, Levelness versus nearly Gorensteinness of homogeneous domains, arXiv:2206.00552. 
\bibitem{M2} S. Miyashita, Nearly Gorenstein projective monomial curves of small codimension, arXiv:2302.04027. 
\bibitem{OH} H. Ohsugi and T. Hibi, Normal polytopes arising from finite graphs, \textit{J. Algebra} \textbf{218}, (1999), 509--527.
\bibitem{S} R. P. Stanley, An introduction to hyperplane arrangements, \textit{Geometric combinatorics}, Amer. Math. Soc., Providence, RI, {\bf 13}, (2007), 389--496. 
\bibitem{SVV} A. Simis, W. V. Vasconcelos and R. H. Villarreal, The integral closure of subrings associated to graphs, \textit{J. Algebra} \textbf{199} (1998), 281--289.
\bibitem{TH} N. V. Trung and L. T. Hoa, Affine semigroups and Cohen-Macaulay rings generated by monomials, \textit{Trans. Amer. Math. Soc.} \textbf{298}, (1986), 145--167.
\end{thebibliography}
\end{document}